\documentclass[12pt,a4paper]{amsart}
\usepackage{fullpage}
\usepackage{ytableau}
\usepackage{mathrsfs}
\usepackage{CJK}
\usepackage{geometry}   
\usepackage{setspace}
\usepackage{amsfonts}
\usepackage{xy}
\usepackage{cite}
\usepackage{amssymb}
\usepackage{color, latexsym}
\usepackage{graphicx}
\usepackage{amsbsy,amscd}
\usepackage{fancyhdr}
\usepackage[colorlinks=true,
           citecolor=blue,
           linkcolor =red]{hyperref}
\xyoption{all}
\allowdisplaybreaks
\makeatletter
\@namedef{subjclassname@2020}{\textup{2020} Mathematics Subject Classification}
\makeatother

\xyoption{all}
\allowdisplaybreaks

\fontsize{9pt}{9pt}\selectfont

\def\Number#1{\refstepcounter{equation}
              \leqno(\theequation)\if*#1%
              \else\def\@currentlabel{{\rm\theequation}}\label{#1}%
              \fi}

\numberwithin{equation}{section}
\swapnumbers
\newtheorem{theorem}[equation]{Theorem}
\newtheorem{lemma}[equation]{Lemma}
\newtheorem{proposition}[equation]{Proposition}
\newtheorem{corollary}[equation]{Corollary}
\theoremstyle{definition}
\newtheorem{definition}[equation]{Definition}
\newtheorem{example}[equation]{Example}
\theoremstyle{remark}
\newtheorem{remark}[equation]{Remark}

\begin{document}
\setlength{\itemsep}{-0.25cm}
\fontsize{12}{\baselineskip}\selectfont
\setlength{\parskip}{0.30\baselineskip}
\vspace*{-4mm}
\title[Brenti--Welker identity]{\fontsize{10}{\baselineskip}\selectfont A categorification
 of the Brenti--Welker identity}
\author{Deke Zhao}
\address{\bigskip\hfil\begin{tabular}{l@{}}
          Department of  Mathematics\\
            Beijing Normal University at Zhuhai, Zhuhai, 519087\\
             China\\
             E-mail: \it deke@amss.ac.cn \hfill
          \end{tabular}}
 \author{Zhankui Xiao}
\address{\bigskip\hfil\begin{tabular}{l@{}}
          School of Mathematical Sciences\\
            Huaqiao University, Quanzhou, 362021\\
             China\\
             E-mail: \it zhkxiao@hqu.edu.cn \hfill
          \end{tabular}}      
\thanks{Zhao is supported by Guangdong Basic and Applied Basic Research Foundation (Grant No. 2023A1515010251) and the National Natural Science Foundation of China (No.~11871107). Xiao is supported by the Natural Science Foundation of Xiamen, China (No. 3502Z202373054). Both authors are supported by the Open Project Program of Key Laboratory of Mathematics and Complex System Beijing Normal University (Grant No. K202402).}

\begin{abstract}The paper aims to provide a categorification of the Brenti--Welker identity involving Eulerian numbers in (Adv. Appl. Math. 42 (2009): 545--556)  by lifting it from an enumerative equality to an isomorphism of symmetric group representations. To do so, we study the decomposition of the tensor product of $(\mathbb{C}^r)^{\otimes n}$ and modules
affording Foulkes characters as modules of the symmetric group.
The main ingredient of the proof is a combinatorial identity which may be of independent interest.
\end{abstract}
\maketitle
\vspace*{-10mm}
\section{Introduction}\label{Sec:Introduction}

The Veronese construction for formal power series are well studied in commutative algebra and
algebraic geometry because of its connection with Hilbert series of a graded algebra, see \cite{Beck-Stapledon,Jo18,Jo22} for example.
It also has an interesting interpretation in the Ehrhart theory of lattice polytopes \cite{Beck-Robins}.
Motivated by the Veronese construction, Brenti and Welker \cite{Brenti-Welker} studied the transformation of the numerator
polynomial of a rational formal power series after taking a subsequence of the coefficients and computing its generating series.
Going further, they described the Veronese construction using two different bases for the ring of formal power series of
some specified form, the one known as ``monomial basis", and the other related to the Eulerian polynomials.
By considering a change of basis for this transformation, they showed an identity involving the
Eulerian numbers, which is now referred as the {\em Brenti--Welker identity}.

To state the Brenti--Welker identity, we need introduce some notation.
Let $n$ be a positive integer and $\mathfrak{S}_n$ be the symmetric group of degree $n$.
Recall that the \emph{Eulerian number} $E_{n}(k)$ counts the number of elements in $\mathfrak{S}_n$ having $k$ descents, see Equ.~(\ref{Equ:Eulerian-number}) below.
We refer the reader to the monograph \cite{P} for a comprehensive introduction to these Eulerian numbers.
For non-negative integers $r,d$, we denote by $c(r,n,d)$ the number of integer solutions of
equation \begin{equation*}
  x_1+x_2+\cdots+x_{n}=d,
\end{equation*}
where $0\leq x_1,\ldots, x_{n}\leq r$. By convention, we write $c(r,0,d)=\delta_{0,d}=c(0,n,d)$.

\begin{theorem}[\protect{Brenti--Welker identity, \cite[Proposition~2.3]{Brenti-Welker}}]\label{Them:BW-identity}
Let $r,n$ be positive integers. Then
\begin{equation*}
  \sum_{k=1}^nc(r-1,n+1,ir-k)E_n(k-1)=r^nE_n(i-1)
\end{equation*}
for $i=1, \ldots,n$. In particular,
\begin{equation*}
 \sum_{k=1}^nc(r-1,n+1,r-k)E_n(k-1)=r^n=\sum_{k=1}^nc(r-1,n+1,nr-k)E_n(k-1).
\end{equation*}
\end{theorem}

Notice that the derivation of the Brenti--Welker identity is in a linear algebra manner in \cite{Brenti-Welker}. Very recently, Valencia-Porras \cite{Porras} presented a combinatorial proof of the Brenti--Welker identity by constructing two
bijections defined on the set of alcoves of the $r$-dilated hypersimplex.

Let $\phi_1^n,\phi_2^n,\ldots, \phi_{n}^n$ be the Foulkes characters of the symmetric group
$\mathfrak{S}_n$ introduced by Foulkes in \cite{Foulkes} and $\Phi_{i}^n$ be the $\mathbb{C}\mathfrak{S}_n$-module
affording the character $\phi_i^n$ ($i=1,\ldots,n$), see Definition~\ref{Def:Foulkes}. Note that $\phi_1^n$ and $\phi_n^n$ are the alternating character and  the trivial character of $\mathfrak{S}_n$ respectively.
Hence $\Phi_1^n$ and $\Phi_n^n$ are the sign representation $\mathrm{sign}_n$ and the trivial representation of $\mathfrak{S}_n$ respectively.
It is well-known that their degrees are Eulerian numbers, i.e.
\begin{equation}\label{Equ:Degree=Euler}
  \phi_i^n(\boldsymbol{e})=E_n(i-1), \text{ for }i=1,\ldots,n,
\end{equation}
where $\boldsymbol{e}$ is the unity of $\mathfrak{S}_n$.
Therefore, the term of the left-hand side of the Brenti--Welker identity is the dimension of certain representation involving $\Phi^n_k$
($k=1,\ldots,n$), while the term of the right-hand side equals the dimension of the $\mathbb{C}\mathfrak{S}_n$-module
$(\mathbb{C}^r)^{\otimes n}\otimes_{\mathbb{C}}\Phi^n_i$.
Inspired by these observations, it is natural to ask for
a categorification of the Brenti--Welker identity, that is, the identity is replaced by an isomorphism of $\mathbb{C}\mathfrak{S}_n$-modules,  which forms our motivation of present work.

The main result of this paper is as follows:

\begin{theorem}\label{Them:Main}Let $r,n$ be positive integers. Then there is a $\mathbb{C}\mathfrak{S}_n$-module isomorphism
\begin{equation*}
\Phi_i^n\otimes_{\mathbb{C}}\left(\mathbb{C}^r\right)^{\otimes n}\cong
\bigoplus_{k=1}^nc(r-1,n+1,ir-k)\Phi_{n+1-k}^n
\end{equation*}for $i=1,\ldots,n$. In particular,
\begin{eqnarray*}
&&\left(\mathbb{C}^r\right)^{\otimes n}\cong
\bigoplus_{k=1}^nc(r-1,n+1,nr-k)\Phi_{n+1-k}^n,\\
&&\mathrm{sign}_n\otimes_{\mathbb{C}}\left(\mathbb{C}^r\right)^{\otimes n}\cong
\bigoplus_{k=1}^nc(r-1,n+1,r-k)\Phi_{n+1-k}^n.
\end{eqnarray*}
\end{theorem}

To prove Theorem~\ref{Them:Main}, we need the following combinatorial identity, which is  a generalization of a fact proved in \cite[Page~147]{Holte} and it may be of independent interest. Throughout the paper, for non-negative integers $m,n$, we set $\tbinom{m}{n}=0$ when $m<n$.

\begin{proposition}\label{Prop:n+ar-k}Let $r,n,a$ be positive integers. Then
\begin{equation*}
\sum_{j=1}^{n}\tbinom{r-1+j}{n}\sum_{b=1}^j
  (-1)^{j-b}\tbinom{n+1}{j-b}\tbinom{ab+k-1}{n}= \tbinom{ar+n-k}{n}
\end{equation*}
 for $1\leq k\leq n$.
\end{proposition}

We would like to remark that Miller \cite{M15} introduced the Foulkes characters of complex reflection groups, which are well-behaved for coincidental complex reflection groups. Hence it is natural and interesting to study the coincidental complex reflection groups versions
of Theorem~\ref{Them:Main} and the Brenti--Welker identity respectively. In particular, the wreath analogue would be useful to
understand the generalized amazing matrices introduced by Nakano and Sadahiro \cite{NS1,NS2,NS3}.

The paper is organized as follows. In Section~\ref{Sec:Preliminaries} we recall some facts related to the representation theory
of symmetric groups, especially on Eulerian numbers and Foulkes characters.
Section~\ref{Sec:Combinatorial-identity} provides a proof of Proposition~\ref{Prop:n+ar-k} by applying the induction argument.
In the last section, we prove our main result Theorem~\ref{Them:Main} and then the Brenti--Welker identity
(Theorem~\ref{Them:BW-identity}). Here, as a byproduct, a remarkable property of the decomposition numbers of the product of Foukles characters
is provided (see Corollary~\ref{Cor:C-diagonalized}).

\section{Preliminaries}\label{Sec:Preliminaries}

This section is devoted to recalling some facts related to the decomposition of the $\mathbb{C}\mathfrak{S}_n$-modules
$\left(\mathbb{C}^r\right)^{\otimes n}\otimes_{\mathbb{C}}\Phi_i^n$ ($i=1,\ldots,n$) in terms of $\Phi_1^n,\ldots,\Phi_n^n$.
Then we also recall a closed formula for $c(r-1,n+1,ir-k)$, where $i,k=1,2,\ldots, n$, see the Introduction for precise definition.

Recall that a  partition $\lambda=(\lambda_1, \lambda_2, \ldots)$ of $n$, denote  $\lambda\vdash n$, is a weakly decreasing sequence  of  nonnegative integers such that $|\lambda|=\sum_{i\geq1}\lambda_i=n$ and we write $\ell(\lambda)$ the length of $\lambda$, i.e. the number of nonzero parts of $\lambda$. It is well-known that the ordinary irreducible representations of $\mathfrak{S}_n$ are parameterized by partitions of $n$. For $\lambda\vdash n$, we denote by $S^{\lambda}$ the Specht module of $\mathfrak{S}_n$ and let $\chi^{\lambda}$ be the character of $S^{\lambda}$.

For an element $\sigma=\sigma(1)\sigma(2)\cdots\sigma(n)$ of the symmetric group $\mathfrak{S}_n$ (in one-line form), we say that $i$ is a \emph{descent} of $\sigma$ if $\sigma(i)>\sigma(i+1)$ and we denote by $\mathrm{Des}(\sigma)$ the set of descents of $\sigma$:
\begin{equation*}
  \mathrm{Des}(\sigma):=\{1\leq i\leq n-1 \mid \sigma(i)>\sigma(i+1)\}.
\end{equation*}

We set $\mathrm{des}(\sigma)$ to be the number of descents (also called the \emph{descent number}) of $\sigma$, that is, $\mathrm{des}(\sigma)=|\mathrm{Des}(\sigma)|$. The \emph{Eulerian number} $E_{n}(k)$ counts the number of elements in $\mathfrak{S}_n$ having $k$ descents:
 \begin{equation}\label{Equ:Eulerian-number}
   E_{n}(k)=\left|\{\sigma\in \mathfrak{S}_n \mid \mathrm{des}(\sigma)=k\}\right|.
 \end{equation}
 For example if $\sigma=3142657\in \mathfrak{S}_7$, $\mathrm{Des}(\sigma)=\{1,3,5\}$.

 According to Foulkes (see e.g. \cite[\S2.1]{D-Fulman}), for any subset $U$ of $\{1,2,\ldots,n-1\}$, we can construct a ribbon shape (also called a rim of Young diagram) $R(U)$ beginning with a single box and
sequentially adding the next box below the last box if $i\in U$, and to the left of the last
box if $i\notin U$.  For example, if $U=\{1,3,5\}\subset\{1,2,3,4,5,6\}$ then the ribbon shape $R(U)$ is constructed as follows:
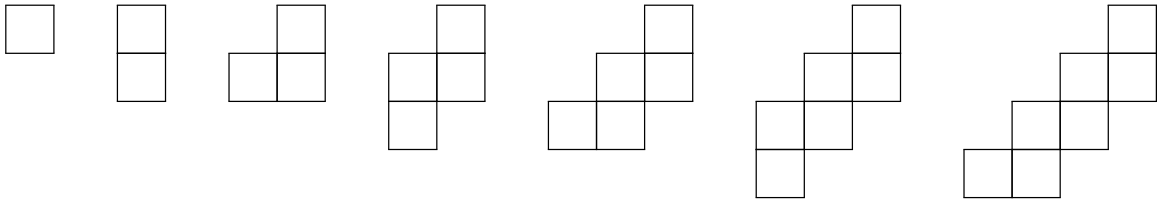
\begin{equation*}
\begin{ytableau}\\ \end{ytableau}\qquad
\begin{ytableau} \\ \\ \end{ytableau}\qquad
\begin{ytableau}\none & \\ &\\ \end{ytableau}\qquad
\begin{ytableau} \none & \\ &\\  \\ \end{ytableau}\qquad
\begin{ytableau}\none&\none&\\\none&&\\&\\\end{ytableau}\qquad
\begin{ytableau}\none&\none&\\\none&&\\&\\ \\\end{ytableau} \qquad
 \begin{ytableau}\none&\none&\none&\\\none&\none&&\\\none&&\\&\\\end{ytableau}
\end{equation*}
The final skew shape $R(U)$ will have $n$ boxes and be the lower rim of a partition $\alpha(U)$; in this example, $\alpha(U)=(4,4,3,2)$ and $R(U)=(4,4,3,2)/(3,2,1)$.
 The skew shape $R(U)$ corresponding to $U\subset \{1,2,\ldots,n-1\}$ gives a skew
character $\chi^{R(U)}$ and a skew representation $S^{R(U)}$ of the symmetric group $\mathfrak{S}_n$. Foulkes showed that the dimension of $S^{R(U)}$ is the number of permutations with given descent set $U$ (see e.g. \cite{Kerber-T}).

\begin{definition}\label{Def:Foulkes}For $i=1,2,\ldots,n$, the \textit{Foulkes character} $\phi_i^n$ of $\mathfrak{S}_n$ is defined as the sum of $\chi^{R(U)}$ over all $U$ with $n-i$ descents. The $\mathfrak{S}_n$-representation $\Phi_i^n$ affording the character $\phi_i^n$ are  defined as the sum of $S^{R(U)}$ over all $U$ with $n-i$ descents.
\end{definition}

Clearly $\phi_1^n$ and $\phi_n^n$ are the alternating character and  trivial character of $\mathfrak{S}_n$ respectively. Thus $\Phi_1^n$ and $\Phi_n^n$ are the sign representation and  trivial representation of $\mathfrak{S}_n$ respectively. Foulkes showed that the dimension of $\phi_i^n$ is
the Eulerian number $E_n(i-1)$ and that $\phi_i^n(\sigma)$ only depends on $\sigma$ through the number of cycles in $\sigma$, that is,  let $\ell_n(\sigma)$ be the number of cycles of a permutation $\sigma\in \mathfrak{S}_n$, then
\begin{equation*}
  \phi_i^n(\tau)=\phi_i^n(\sigma) \text{ whenever }\ell_n(\tau)=\ell_n(\sigma),
\end{equation*}
where $i=1, \ldots, n$. For the history and properties of Foulkes characters, see \cite[Chapter~8.5]{Kerber}.

Let us remark that the decomposition of the Foulkes characters can be obtain by applying the classical Schur--Weyl duality and $\Phi_i^n$ ($1\leq i\leq n$) can be constructed via the coinvariant algebra of $\mathfrak{S}_n$ (see \cite[Proposition~2.5 and Theorem~4.4]{GGK}).

Now let $\mathrm{CF}_{\ell_n}(\mathfrak{S}_n)$ be the space  of class functions on $\mathfrak{S}_n$ that depend only on the length $\ell_n$. Then $\phi_1^n,\ldots, \phi_n^n$
form a basis for $\mathrm{CF}_{\ell_n}(\mathfrak{S}_n)$, see e.g. \cite[Theorem~8.5]{Kerber}.
It is clear that the product $ \phi_i^n\phi_j^n$ is still a class function of $\mathfrak{S}_n$ depending only on the length $\ell_n$
and hence we write
\begin{equation}\label{Equ:Foulkes-product-decom}
 \phi_i^n\phi_j^n=\sum_{k=1}^nc_{i,j,k}^n\phi_k^n
\end{equation}
for $1\leq i,j,k\leq n$. The coefficients $c_{i,j,k}^n$ can be calculated by the following closed formula.

\begin{theorem}[\protect{\cite[Theorem~2.8]{M21}}]\label{Them:decom-numbers}Let $n$ be a positive integer. Then
 \begin{equation*}
  c_{i,j,k}^n=\sum_{a=1}^i\sum_{b=1}^j
  (-1)^{i+j-a-b}\tbinom{n+1}{i-a}\tbinom{n+1}{j-b}\tbinom{n+ab-k}{n}
 \end{equation*}
 for $i,j,k=1,\ldots,n$.
\end{theorem}

Let $\chi_r^n$ be the ``P\'{o}lya character" affording the permutation module $\left(\mathbb{C}^r\right)^{\otimes n}$ of $\mathfrak{S}_n$. Then
\begin{equation}\label{Equ:Polya-character}
  \chi_r^n(\sigma)=r^{\ell_n(\sigma)} \text{ for $\sigma\in \mathfrak{S}_n$}.
\end{equation}
Moreover, we have (see \cite[\S2]{Kerber-T} or \cite[Corollary~8.5.11]{Kerber})
\begin{equation}\label{Equ:chi=Foulkes}
  \chi^n_r=\sum_{j=1}^n\tbinom{r-1+j}{n}\phi^n_{j}
\end{equation}
for any positive integer $r$. Notice that two $\mathbb{C}\mathfrak{S}_n$-modules are isomorphic if and only if
their characters coincide with each other. Thus (\ref{Equ:chi=Foulkes}) shows
\begin{equation}\label{Equ:Rep=chi=Foulkes}
  \left(\mathbb{C}^r\right)^{\otimes n}\cong\bigoplus_{j=1}^n\tbinom{r-1+j}{n}\Phi^n_{j}.
\end{equation}

For $i,j=1,\ldots,n$, thanks to Equ.~\eqref{Equ:Foulkes-product-decom}, we have the isomorphism of $\mathbb{C}\mathfrak{S}_n$-modules
\begin{equation}\label{Equ:Decom-Foulkes-modules-product}
 \Phi_i^n\otimes_{\mathbb{C}}\Phi _j^n\cong
 \bigoplus_{k=1}^nc_{i,j,k}^n\Phi_k^n
\end{equation}

Combining Equs.~(\ref{Equ:Rep=chi=Foulkes}) and (\ref{Equ:Decom-Foulkes-modules-product}), we yield that the following fact.
\begin{corollary}\label{Cor:iso-modules}For $i=1,\ldots,n$,
  \begin{eqnarray*}
 &&\Phi_i^n\otimes_{\mathbb{C}}\left(\mathbb{C}^r\right)^{\otimes n}
 \cong\bigoplus_{k=1}^{n}\sum_{j=1}^nc_{i,j,k}^n\tbinom{r-1+j}{n}\Phi^n_{k}.
\end{eqnarray*}
In particular, for $i=1, \ldots, n$,
\begin{eqnarray*}
 &&\sum_{k=1}^{n}\sum_{j=1}^nc_{i,j,k}^n\tbinom{r-1+j}{n}E_n(k-1)=r^nE_{n}(i-1).
\end{eqnarray*}
\end{corollary}

Let $r,n$ be positive integers. In the remainder part of this section,
we introduce some combinatorics for $c(r-1,n+1,ir-k)$ for $1\leq i,k\leq n$ via generating functions.
It follows from the definition that $c(r-1,n+1,ir-k)$ is the number of integer solutions of equation
\begin{equation*}
  x_1+x_2+\cdots+x_{n+1}=ir-k,
\end{equation*}
where $0\leq x_1,\ldots, x_{n+1}\leq r-1$, which is the
coefficient of $x^{ir-k}$ in $(1+x+x^2+\cdots+x^{r-1})^{n+1}$. Since
\begin{equation*}
 (1+x+x^2+\cdots+x^{r-1})^{n+1}=(1-x^{r})^{n+1}(1-x)^{-n-1}
\end{equation*}
and
\begin{equation*}
 (1-x^{r})^{n+1}=\sum_{a=0}^{n+1}\tbinom{n+1}{a}(-x^{r})^{a}, \quad (1-x)^{-(n+1)}=\sum_{a=0}^{\infty}\tbinom{n+a}{n}x^{a},
\end{equation*}
we yield
\begin{equation*}
 c(r-1,n+1,ir-k)=\sum_{a=0}^{i-[k/r]}(-1)^a\tbinom{n+1}{a}\tbinom{n+(i-a)r-k}{n}.
\end{equation*}
Notice that $\tbinom{m}{n}=0$ whenever $m<n$, the above equality can be rewritten as follows
\begin{equation}\label{Equ:c(r-1,n+1,ir-k)}
 c(r-1,n+1,ir-k)=\sum_{a=0}^{i-1}(-1)^a\tbinom{n+1}{a}\tbinom{(i-a)r+n-k}{n}.
\end{equation}

Comparing Corollary~\ref{Cor:iso-modules} with Theorem~\ref{Them:Main}, it is natural to expect that, for $1\leq i,k\leq n$, we have
\begin{eqnarray*}
  c(r-1,n+1,ir-k)&=&\sum_{j=1}^{n}c_{i,j,k}^{n}\tbinom{r-1+j}{n}\\
  &=&\sum_{j=1}^{n}\sum_{a=1}^i\sum_{b=1}^j  (-1)^{i+j-a-b}\tbinom{n+1}{i-a}\tbinom{n+1}{j-b}\tbinom{n-k+ab}{n}\tbinom{r-1+j}{n}.
\end{eqnarray*}
Unfortunately, the follow examples show that it is not true, which motivate us to expect that  the following equality holds (see Lemma~\ref{Lemm:c=r-c}):  for $i,k=1,\ldots,n$, we have \begin{eqnarray*}
  c(r-1,n+1,ir-k)&=&\sum_{j=1}^{n}c_{i,j,n+1-k}^{n}\tbinom{r-1+j}{n}\\
  &=&\sum_{j=1}^{n}\sum_{a=1}^i\sum_{b=1}^j
  (-1)^{i+j-a-b}\tbinom{r-1+j}{n}
  \tbinom{n+1}{i-a}\tbinom{n+1}{j-b}\tbinom{ab+k-1}{n}.
 \end{eqnarray*}

\begin{example}\label{Exam:n=2}Let $n=2$ and $r$ be a positive integer. Then
 \begin{eqnarray*}
&&c(r-1,3,r-1)=\sum_{j=1}^{2}\tbinom{r-1+j}{2}c_{1,j,2}^{2}=\tbinom{r+1}{2}=
\sum_{j=1}^{2}\tbinom{r-1+j}{2}c_{2,j,1}^{2}=c(r-1,3,2r-2);\\
&&c(r-1,3,r-2)=\sum_{j=1}^{2}c_{1,j,1}^{2}\tbinom{r-1+j}{2}=\tbinom{r}{2}=
\sum_{j=1}^{2}\tbinom{r-1+j}{2}c_{2,j,2}^{2}=c(r-1,3,2r-1).
\end{eqnarray*}
\end{example}

\begin{example}\label{Exam:n=3}Let $n=3$ and $r$ be a positive integer. Then
 \begin{eqnarray*}
c(r-1,4,r-1)&=&\sum_{j=1}^{3}c_{1,j,3}^{3}\tbinom{r-1+j}{3}=\tbinom{r+2}{3};\\
c(r-1,4,r-2)&=&\sum_{j=1}^{3}c_{1,j,2}^{3}\tbinom{r-1+j}{3}=\tbinom{r+1}{3};\\
c(r-1,4,r-3)&=&\sum_{j=1}^{3}c_{1,j,1}^{3}\tbinom{r-1+j}{3}=\tbinom{r}{3};\\
c(r-1,4,2r-1)&=&\sum_{j=1}^{3}c_{2,j,3}^{3}\tbinom{r-1+j}{3}=4\tbinom{r+1}{3};\\
c(r-1,4,2r-2)&=&\sum_{j=1}^{3}c_{2,j,2}^{3}\tbinom{r-1+j}{3}=4\tbinom{r+1}{3}+r;\\
c(r-1,4,2r-3)&=&\sum_{j=1}^{3}c_{2,j,1}^{3}\tbinom{r-1+j}{3}=4\tbinom{r+1}{3};\\
c(r-1,4,3r-1)&=&\sum_{j=1}^{3}c_{3,j,3}^{3}\tbinom{r-1+j}{3}=\tbinom{r}{3};\\
c(r-1,4,3r-2)&=&\sum_{j=1}^{3}c_{3,j,2}^{3}\tbinom{r-1+j}{3}=\tbinom{r+1}{3};\\
c(r-1,4,3r-3)&=&\sum_{j=1}^{3}c_{3,j,1}^{3}\tbinom{r-1+j}{3}=\tbinom{r+2}{3}.
\end{eqnarray*}
\end{example}

\section{Proof of Proposition~\ref{Prop:n+ar-k}}\label{Sec:Combinatorial-identity}

This section is devoted to proving the combinatorial identity in Proposition~\ref{Prop:n+ar-k} by induction arguments.

Let us begin with the $(n+1)$-th finite difference of polynomials. Suppose that $p(x)$ is a polynomial.
Then the $(n+1)$-th finite difference of $p(x)$ is defined recursively by
\begin{equation*}
 \vartriangle^{n+1}\!p(x)=\vartriangle(\vartriangle^{n} p(x)),
\end{equation*}
where  $\vartriangle\!p(x):=p(x+1)-p(x)$ is the first difference of $p(x)$. It is well-known that
\begin{equation}\label{Equ:n-difference}
  \vartriangle^{n+1}\!p(x)=\sum_{k=0}^{n+1}(-1)^{k}\tbinom{n+1}{k}p(x+n+1-k)
\end{equation}
and $\vartriangle^{n+1}\!p(x)=0$ whenever $\mathrm{deg\,}p(x)<n+1$.

For $\ell, s=0,1, \ldots, n-1$, we let $p(x)=x^{n-s}$.
Then Equ.~(\ref{Equ:n-difference}) implies that
\begin{equation*}
 \sum_{k=0}^{n+1}(-1)^{k}\tbinom{n+1}{k}(n+1+x-k)^{n-s}=0.
\end{equation*}
In particular, letting $x=\ell-n$, we yield that
\begin{equation}\label{Equ:n+1-diff=zero}
 \sum_{k=0}^{n+1}(-1)^{k}\tbinom{n+1}{k}(\ell+1-k)^{n-s}=0
\end{equation}
for $\ell, s=0,1,\ldots, n-1$.

\begin{lemma}\label{Lemm:s=n-s}Let $n$ be a positive integer. Then
\begin{equation*}
  \sum_{i=0}^{\ell}(-1)^{i}\tbinom{n+1}{i}(\ell+1-i)^{n-s}=
  (-1)^s\sum_{b=1}^{n-\ell}(-1)^{n-\ell-b}\tbinom{n+1}{n-\ell-b}b^{n-s}
\end{equation*}
for $\ell, s=0,\ldots,n-1$. In particular,
\begin{equation*}
  \sum_{b=1}^{n}(-1)^{n-b}\tbinom{n+1}{n-b}b^{n-s}=(-1)^s
\end{equation*}
for $s=0,\ldots,n-1$.
\end{lemma}

\begin{proof}
Thanks to Equ.~(\ref{Equ:n+1-diff=zero}), we have
\begin{eqnarray*}
 \sum_{i=0}^{\ell}(-1)^{i}\tbinom{n+1}{i}(\ell+1-i)^{n-s}
 &=&-\sum_{i=\ell+1}^{n+1}(-1)^{i}\tbinom{n+1}{i}(\ell+1-i)^{n-s}\\
 &=&(-1)^{n+\ell-s}\sum_{j=0}^{n-\ell}(-1)^j\tbinom{n+1}{n-\ell-j}j^{n-s}\\
 &=&(-1)^{s}\sum_{k=0}^{n-\ell}(-1)^k\tbinom{n+1}{k}(n-\ell-k)^{n-s},
\end{eqnarray*}
where the second and the last equalities follows by substituting $j=i-(\ell+1)$ and $k=n-\ell-j$ respectively.

On the other hand, by substituting $k=n-\ell-b$, we have
\begin{eqnarray*}
 (-1)^s\sum_{b=1}^{n-\ell}(-1)^{n-\ell-b}\tbinom{n+1}{n-\ell-b}b^{n-s}
 &=&(-1)^s\sum_{k=0}^{n-\ell-1}(-1)^{k}\tbinom{n+1}{k}(n-\ell-k)^{n-s}\\
 &=&(-1)^s\sum_{k=0}^{n-\ell}(-1)^{k}\tbinom{n+1}{k}(n-\ell-k)^{n-s}.
\end{eqnarray*}
Combining the above two identities, we complete the proof of the lemma.
\end{proof}

\begin{lemma}\label{Lemm:r=1}Let $a,n$ be positive integers. Then
\begin{equation*}
\sum_{b=1}^n (-1)^{n-b}\tbinom{n+1}{n-b}\tbinom{ab+k-1}{n}=\tbinom{a+n-k}{n}
\end{equation*}
 for $k=1,\ldots,n$.
\end{lemma}

\begin{proof}
For a fixed integer $k$ with $1\leq k\leq n$, we set
\begin{equation*}
 f_n(a):=\sum_{b=1}^n (-1)^{n-b}\tbinom{n+1}{n-b}\tbinom{ab+k-1}{n}-\tbinom{a+n-k}{n}.
\end{equation*}
Then it is easy to see that $f_n(a)$ is a polynomial of degree $n$ in variable $a$ with constant term zero.
Let $e_s$ be the $s$-th elementary symmetric polynomial of $1-k, 2-k, \ldots, n-k$ for $s=0, 1,\ldots, n-1$ with convention $e_0=1$.
If we write $f_n(a)=\sum_{s=0}^{n-1}c_sa^{n-s}$, then Lemma~\ref{Lemm:s=n-s} shows that
\begin{equation*}
 c_s=\left((-1)^s\sum_{b=1}^n(-1)^{n-b}\tbinom{n+1}{n-b}b^{n-s}-1\right)e_s=0
\end{equation*}
for $s=0, 1,\ldots,n-1$. Thus the lemma is proved.\end{proof}

Now we can prove Proposition~\ref{Prop:n+ar-k} by induction arguments.

\begin{proof}[Proof of Proposition~\ref{Prop:n+ar-k}]
For an arbitrary fixed positive integer $a$ and $k$ with $1\leq k\leq n$,
we set
\begin{equation*}G(r;a,n):=\tbinom{ar+n-k}{n}-\sum_{j=1}^{n}\tbinom{r-1+j}{n}\sum_{b=1}^j
  (-1)^{j-b}\tbinom{n+1}{j-b}\tbinom{ab+k-1}{n}.
\end{equation*}
Then $G(r;a,n)$ is a polynomial in variable $r$ of degree $n$. Clearly $G(0;a,n)=0$.
Thus it suffices to show that $G(r;a,n)=0$ for $r=1, 2,\ldots,n$. We show that  $G(r;a,n)=0$  by applying the induction argument on $r$.
For $r=1$, Lemma~\ref{Lemm:r=1} shows that $G(1;a,n)=0$. Now assume that $G(r;a,n)=0$ for $r=1,\ldots, \ell<n$.
We need to show that $G(\ell+1;a,n)=0$.

For $j=1,\ldots, n$, we set
\begin{equation*}
  g(j;a,n)=\sum_{b=1}^j
  (-1)^{j-b}\tbinom{n+1}{j-b}\tbinom{ab+k-1}{n}.
\end{equation*}
By the induction hypotheses, we have $G(t;a,n)=0$ for $t=1,\ldots,\ell$, in other words,
\begin{eqnarray*}
 &&\displaystyle\sum_{i=1}^{t}\tbinom{n+1}{t-i}g(n+1-i;a,n)=\tbinom{ta+n-k}{n}
\end{eqnarray*}
for $t=1,\ldots,\ell$. As a consequence, we yield that
\begin{equation}\label{Equ:g(j;a,n)}
 g(n+1-t;a,n)=\sum_{i=0}^{t-1}\tbinom{n+1}{i}\tbinom{(t-i)a+n-k}{n}
\end{equation}
for $t=1,\ldots,\ell$.

Now by applying Equ.~(\ref{Equ:g(j;a,n)}) repeatedly, we yield that
\begin{eqnarray}\label{Equ:G-ell+1}
&&G(\ell+1;a,n)=\sum_{i=0}^{\ell}(-1)^{i}\tbinom{n+1}{i}\tbinom{a(\ell+1-i)+n-k}{n}-g(n-\ell;a,n).
\end{eqnarray}
Notice that the constant term of $G(\ell+1;a,n)$ is zero, and thus we can write
\begin{eqnarray*}G(\ell+1;a,n)
&=&\sum_{s=1}^{n}c_{\ell,s}a^{n-s}.
\end{eqnarray*}
Then, for $s=0,1,\ldots, n-1$, Equ.~(\ref{Equ:G-ell+1}) shows
\begin{eqnarray*}
 c_{\ell,s}=\left(\sum_{i=0}^{\ell+1}(-1)^i\tbinom{n+1}{i}(\ell+1-i)^{n-s}-
 (-1)^s\sum_{b=1}^{n-\ell}(-1)^{n-\ell-b}\tbinom{n+1}{n-\ell-b}b^{n-s}\right)e_s,
\end{eqnarray*}
where $e_s$ is the $s$-th elementary symmetric polynomial of $1-k, 2-k, \ldots,n-k$ with $e_0=1$.
Applying Lemma~\ref{Lemm:s=n-s}, we can see that $c_{\ell,s}=0$ for $s=0,1,\ldots, n-1$,
that is, $G(\ell+1;a,n)=0$. This completes the proof of the proposition.
\end{proof}

In \cite[Page~87]{GGK}, Gnedin, Gorin and Kerov remarked that the Foulkes characters $\phi_i^n$ ($i=1,\ldots,n$) can be obtained by the iterated difference of the sequence $\chi_{\centerdot}^n$ (with $\chi_i^n=0$ for $i\leq 0$).  Specifically, let $\triangledown$ be the forward difference operators acting as $\triangledown(x_{\centerdot})_i=x_i-x_{i-1}$. Then $\phi_i^n=\triangledown^n(\chi_{\centerdot}^n)_i$ for $i=1,\ldots,n$. It may be interesting to clarify the relationship between $\phi_i^n=\triangledown^n(\chi_{\centerdot}^n)_i$ and Proposition~\ref{Prop:n+ar-k}.


The following fact follows directly by letting $r=1$ and $k=1$ in Proposition~\ref{Prop:n+ar-k}, which is a generalization of a result proved in \cite[Page~147]{Holte}.
\begin{corollary}Let $a, n$ be positive integers. Then
\begin{equation*}
  \sum_{b=1}^n(-1)^{n-b}\tbinom{n+1}{n-b}\tbinom{ab}{n}=\tbinom{n+a-1}{n}.
\end{equation*}
In particular,
\begin{equation*}
  \sum_{b=1}^j(-1)^{b-1}\tbinom{n+1}{b-1}\tbinom{n+1-i-j-b}{n}=\delta_{ij}
\end{equation*}
for $i,j=1,\ldots, n$.
\end{corollary}

\section{Proof of Theorem~\ref{Them:Main} and some applications}\label{Sec:Applications}

This section devotes to prove the main result Theorem~\ref{Them:Main} and the Brenti--Welker identity
Theorem~\ref{Them:BW-identity}. Furthermore, as a byproduct, we
show a property of the decomposition numbers of the products of Foukles characters.

The following crucial fact is inspired by Examples~\ref{Exam:n=2} and \ref{Exam:n=3}, which
provides the right way to the Brenti--Welker identity, i.e., to prove Theorem~\ref{Them:Main} by Proposition~\ref{Prop:n+ar-k}.

\begin{lemma}\label{Lemm:c=r-c}
Let $r,n$ be positive integers. Then
\begin{eqnarray*}
  c(r-1,n+1,ir-k)&=&\sum_{j=1}^{n}c_{i,j,n+1-k}^{n}\tbinom{r-1+j}{n}
 \end{eqnarray*}
  for $i,k=1,\ldots,n$, that is,
 \begin{eqnarray*}
  c(r-1,n+1,ir-k)&=&\sum_{j=1}^{n}\sum_{a=1}^i\sum_{b=1}^j
  (-1)^{i+j-a-b}\tbinom{r-1+j}{n}
  \tbinom{n+1}{i-a}\tbinom{n+1}{j-b}\tbinom{ab+k-1}{n}
\end{eqnarray*}
for $i,k=1,\ldots,n$.
\end{lemma}

\begin{proof}It follows from Theorem~\ref{Them:decom-numbers} that
\begin{eqnarray*}
  \sum_{j=1}^{n}c_{i,j,n+1-k}^{n}\tbinom{r-1+j}{n}&=&
  \sum_{j=1}^{n}\sum_{a=1}^i\sum_{b=1}^j(-1)^{i+j-a-b}\tbinom{r-1+j}{n}  \tbinom{n+1}{i-a}\tbinom{n+1}{j-b}\tbinom{ab+k-1}{n}\\
  &=&\sum_{a=1}^i(-1)^{i-a}\tbinom{n+1}{i-a}
  \sum_{j=1}^{n}\sum_{b=1}^j(-1)^{j-b}
  \tbinom{r-1+j}{n}\tbinom{n+1}{j-b}\tbinom{ab+k-1}{n}\\
  &=&\sum_{a=1}^i(-1)^{i-a}\tbinom{n+1}{i-a}\tbinom{n+ar-k}{n}\quad (\text{by Proposition~\ref{Prop:n+ar-k}})\\
  &=&\sum_{t=0}^{i-1}(-1)^t\tbinom{n+1}{t}\tbinom{n+(i-t)r-k}{n}\\
  &=&c(r-1,n+1,ir-k)\quad (\text{by Equ.~(\ref{Equ:c(r-1,n+1,ir-k)})}),
 \end{eqnarray*}
which completes the proof of the lemma.
\end{proof}

\begin{remark}
Lemma~\ref{Lemm:c=r-c} may enable us to give an alternative closed formula for the decomposition numbers $c_{i,j,k}^n$
of products of Foulkes characters in terms of $c(r-1,n+1,ir-k)$. More precisely, for $i,k=1,2,\ldots,n$, Lemma~\ref{Lemm:c=r-c} shows
\begin{equation*}
  \left(\begin{array}{cccc}&&&\tbinom{n}{n}\\
  &&\tbinom{n}{n}&\tbinom{n+1}{n}\\
  &{\mathinner{\mkern2mu\raise1pt\hbox{.}\mkern2mu
\raise4pt\hbox{.}\mkern2mu\raise7pt\hbox{.}\mkern1mu}}&\vdots&\vdots\\
  \tbinom{n}{n}&\cdots&\tbinom{2n-2}{n}&\tbinom{2n-1}{n}\\
  \end{array}\right)
  \left(\begin{array}{c}
  c^n_{i,1,n+1-k}\\c^n_{i,2,n+1-k}\\\vdots\\c^n_{i,n,n+1-k}
  \end{array}
  \right)=\left(\begin{array}{c}
  c(0,n+1,i-k)\\c(1,n+1,2i-k)\\\vdots\\c(n-1,n+1,ni-k)
  \end{array}
  \right),
\end{equation*}
which gives a closed formula for $c_{i,j,k}^n$ in terms of $c(r-1,n+1,ir-k)$.
\end{remark}

Now we are ready to prove Theorem~\ref{Them:Main}.

\begin{proof}[Proof of Theorem~\ref{Them:Main}]For $i=1,\ldots, n$ and $r\geq 1$, Equ.~(\ref{Equ:chi=Foulkes}) shows that
\begin{eqnarray*}
\Phi_i^n\otimes_{\mathbb{C}}\left(\mathbb{C}^r\right)^{\otimes n}&=&\sum_{j=1}^n\tbinom{r-1+j}{n}\Phi_i^n\otimes_{\mathbb{C}}\Phi_j^n\\
&=&\sum_{k=1}^n\sum_{j=1}^n\tbinom{r-1+j}{n}c_{i,j,k}^n\Phi_k^n\\
&=&\sum_{k=1}^nc(r-1,n+1, ir-(n+1-k))\Phi_k^n\\
&=&\sum_{k=1}^nc(r-1,n+1, ir-k)\Phi_{n+1-k}^n.
\end{eqnarray*}
Notice that $\Phi_1^n$ and $\Phi_n^n$ are the sign representation $\mathrm{sign}_n$ and the trivial representation of
$\mathbb{C}\mathfrak{S}_n$ respectively. Thus we have
\begin{eqnarray*}
&&\left(\mathbb{C}^r\right)^{\otimes n}\cong
\bigoplus_{k=1}^nc(r-1,n+1,nr-k)\Phi_{n+1-k}^n,\\
&&\mathrm{sign}_n\otimes_{\mathbb{C}}\left(\mathbb{C}^r\right)^{\otimes n}\cong
\bigoplus_{k=1}^nc(r-1,n+1,r-k)\Phi_{n+1-k}^n.
\end{eqnarray*}
We complete the proof.
\end{proof}

We now prove the Brenti--Welker identity (see Theorem~\ref{Them:BW-identity}).

\begin{proof}[Proof of Theorem~\ref{Them:BW-identity}]Notice that $\phi_i^n(\boldsymbol{e})=E_n(i-1)$ and $\chi_r^n(\boldsymbol{e})=r^n$.
In particular, $\phi_n^n(\boldsymbol{e})=1$. By Theorem~\ref{Them:Main}, we have
\begin{eqnarray*}
r^nE_n(i-1)
&=&\mathrm{dim}_{\mathbb{C}}(\Phi_i^n\otimes_{\mathbb{C}}(\mathbb{C}^r)^{\otimes n})\\
&=&\sum_{k=1}^nc(r-1, n+1,ir-k)\phi_{n+1-k}^n(\boldsymbol{e})\\
&=&\sum_{k=1}^nc(r-1, n+1, ir-k)\phi_{k}^n(\boldsymbol{e})\\
&=&\sum_{k=1}^nc(r-1, n+1, ir-k)E_n(k-1),
\end{eqnarray*}
where the third equality follows from the palindromic structure of the sequence of Eulerian numbers
(see \cite[(4.1)]{P} for example), that is, $\phi_{k}^n(\boldsymbol{e})=\phi_{n+1-k}^n(\boldsymbol{e})$ for $k=1, \ldots, n$.
In particular, since $E_n(0)=E_n(n-1)=1$, we can get
 \begin{eqnarray*}
&& \sum_{k=1}^nc(r-1,n+1,r-k)E_n(k-1)=r^n=\sum_{k=1}^nc(r-1,n+1,nr-k)E_n(k-1).
\end{eqnarray*}
The proof of the theorem is completed.
\end{proof}

Now we present a property of the decomposition numbers $c_{i,j,k}^n$ of the products of Foukles characters. For a positive integer $i$, let $\mathfrak{S}_i$ be the symmetric group of degree $i$.
Recall that the \emph{Eulerian polynomial} $E_i(t)$ is defined as follows
\begin{equation*}
  E_i(t)=\sum_{\sigma\in \mathfrak{S}_i}t^{\mathrm{des}(\sigma)}=\sum_{k=1}^iE_i(k-1)t^{k-1}.
\end{equation*}
Let $\mathbb{Q}_n$ be the $\mathbb{Q}$-vector space of all polynomials in variable $t$ of degree less than $n$.
Then $\mathbb{Q}_n$ has two bases $\mathfrak{B}_n^1=\{t^i \mid 0\leq i<n\}$ and
$\mathfrak{B}_n^2=\{E_{i}(t)(1-t)^{n-i} \mid 1\leq i\leq n\}$.
Let $\mathbf{Z}_n=(z^n_{i,j})_{1\leq i,j\leq n}$ be the $n\times n$ matrix with entries define by the following equality of polynomials
\begin{equation}\label{Equ:Matrix}
 E_{i}(t)(1-t)^{n-i}=\sum_{k=1}^{n}z^n_{j,i}t^{j-1}.
\end{equation}
Then $\boldsymbol{Z}_n$ is the base change matrix between $\mathfrak{B}_n^1$ and $\mathfrak{B}_n^2$, which is invertible and
\begin{equation*}
  z_{i,j}^n=\sum_{d=1}^i(-1)^{i-d}\tbinom{n-j}{i-d}E_j(d-1).
\end{equation*}
Notice that there is an explicit formula for Eulerian numbers (see \cite[Corollary~1.3]{P}):
\begin{equation*}
  E_j(d-1)=\sum_{m=1}^{d}(-1)^{m-1}\tbinom{j+1}{m-1}(d+1-m)^j
\end{equation*}
for any positive integers $d,j$ with $1\leq d\leq j$. Thus
\begin{equation*}
  z_{i,j}^n=(-1)^i\sum_{d=1}^i\sum_{m=1}^d(-1)^{m}\tbinom{n-j}{i-d}\tbinom{j+1}{m}m^j.
\end{equation*}
It is easy to see that $z_{i,n}^n=E_n(i-1)$ and $z_{1,i}^n=1$ for $i=1,\ldots, n$.

Further, for any positive integer $r$, we also let $\boldsymbol{C}_{n,r}=\left(c_{i,k}^{n,r}\right)_{1\leq i,k\leq n}$
be the $n\times n$ matrix with entries
\begin{equation*}
  c_{i,k}^{n,r}=\sum_{j=1}^{n}c_{i,j,n+1-k}^{n}\tbinom{r-1+j}{n}.
\end{equation*}
Now Theorem~\ref{Them:decom-numbers} shows $\boldsymbol{C}_{n,r}$ is a center-symmetric matrix, that is, $c_{i,k}^{n,r}=c_{n+1-i,n+1-k}^{n,r}$ for $i,k=1,\ldots,n$.

\begin{example}For any positive integer $r$, we have
\begin{align*}
\boldsymbol{Z}_2=\left(\begin{array}{cc}
  1&1 \\-1&1
  \end{array}\right),\quad
 \boldsymbol{C}_{2,r}=\left(\begin{array}{cc}\vspace{0.5\jot}
  \tbinom{r+1}{2}&\tbinom{r}{2}\\ \tbinom{r}{2}&\tbinom{r+1}{2}
  \end{array}\right),\quad
\end{align*}
and $\boldsymbol{Z}_2^{-1}\boldsymbol{C}_{2,r}\boldsymbol{Z}_2=\mathrm{diag}(r,r^2)$.
\end{example}

\begin{example}For any positive integer $r$, we have
\begin{align*}\setlength{\itemsep}{4\jot}
\boldsymbol{Z}_3=\left(\begin{array}{ccc}
  1&1&1\\-2&0&4\\1&-1&1
  \end{array}\right),\quad
  \boldsymbol{C}_{3,r}=\left(\begin{array}{ccc}\vspace{0.5\jot}
  \tbinom{r+2}{3}&\tbinom{r+1}{3}&\tbinom{r}{3}\\\vspace{0.5\jot}
  4\tbinom{r+1}{3}&4\tbinom{r+1}{3}+r&4\tbinom{r+1}{3}\\
  \tbinom{r}{3}&\tbinom{r+1}{3}&\tbinom{r+2}{3}
  \end{array}\right), \quad
\end{align*}
and $\boldsymbol{Z}_3^{-1}\mathbf{C}_{3,r}\boldsymbol{Z}_3=\mathrm{diag}(r,r^2,r^3)$.
\end{example}

Combing Lemma~\ref{Lemm:c=r-c} and \cite[Lemma~2.1]{Brenti-Welker}, we obtain the following fact:
\begin{corollary}\label{Cor:C-diagonalized}For any positive integers $r, n$, we have
\begin{equation*}
\boldsymbol{Z}^{-1}_n\boldsymbol{C}_{n,r}\boldsymbol{Z}_{n}=\mathrm{diag}(r, r^2, \ldots, r^{n}).
\end{equation*}
\end{corollary}
We end this paper with some remarks.

\begin{remark}
 It is natural to ask for a direct proof of Corollary~\ref{Cor:C-diagonalized},
which in turn can provide an alternative proof of Lemma~\ref{Lemm:c=r-c}. Of course, we may rewrite the identity as a polynomial
of $r$ of degree not greater than $n$, and we only need to show that it holds for $r=0,1,\ldots,n$. In other words, it suffices to show that
\begin{eqnarray*}
 \mathcal{\widetilde{Z}}^n_{i,j}(r)&=&\sum_{k=1}^n\sum_{p=1}^{r}c_{i,p,n+1-k}^{n}
 \tbinom{n+r-p}{r-p}z_{k,j}^n-r^jz_{i,j}^n=0
\end{eqnarray*}
for $r=0, 1,\ldots,n$. However, for us it seems very difficult to prove.
\end{remark}

Recall that $\chi_r^n$ is the ``Polya character" of $\mathfrak{S}_n$ (see Equ.~\eqref{Equ:Polya-character}) and note that $\chi_1^n, \ldots, \chi_n^n$ is a $\mathbb{Q}$-basis of $\mathrm{CF}_{\ell_n}(\mathfrak{S}_n)$. Thanks to Equ.~\eqref{Equ:chi=Foulkes}, we yield
\begin{equation*}
  \left(\begin{array}{cccc}&&&\tbinom{n}{n}\\
  &&\tbinom{n}{n}&\tbinom{n+1}{n}\\
  &{\mathinner{\mkern2mu\raise1pt\hbox{.}\mkern2mu
\raise4pt\hbox{.}\mkern2mu\raise7pt\hbox{.}\mkern1mu}}&\vdots&\vdots \\
    \tbinom{n}{n}&\cdots&\tbinom{2n-2}{n}&\tbinom{2n-1}{n}\\
  \end{array}\right)
  \left(\begin{array}{c}
  \phi_1^n\\\phi_2^n\\\vdots\\\phi_n^n
  \end{array}
  \right)=\left(\begin{array}{c}
 \chi_1^n\\\chi_2^n\\\vdots\\\chi_n^n
  \end{array}
  \right),
\end{equation*}
or equivalently,
\begin{equation*}
    \left(\begin{array}{c}
  \phi_n^n\\\phi_{n-1}^n\\\vdots\\\phi_1^n
  \end{array}
  \right)=\left(\begin{array}{cccc}\tbinom{n+1}{0}&&&\\
  -\tbinom{n+1}{1}&\tbinom{n+1}{0}&&\\
  \vdots&\ddots&\ddots& \\
    (-1)^{n-1}\tbinom{n+1}{n-1}&\cdots&-\tbinom{n+1}{1}&\tbinom{n+1}{0}\\
  \end{array}\right)\left(\begin{array}{c}
 \chi_1^n\\\chi_2^n\\\vdots\\\chi_n^n
  \end{array}
  \right).
\end{equation*}
Finally let $\mathcal{C}_j=\{\sigma\in \mathfrak{S}_n|\ell_n(\sigma)=j\}$ and let  $\Phi=(\phi_{n+1-i}^n(\mathcal{C}_{n+1-j}))_{1\leq i,j\leq n}$ be (reduced) Foulkes character table of $\mathfrak{S}_n$. Then
\begin{equation}\label{Equ:Vandmonde}
  \Phi=\left(\begin{array}{cccc}\tbinom{n+1}{0}&&&\\
  -\tbinom{n+1}{1}&\tbinom{n+1}{0}&&\\
  \vdots&\ddots&\ddots& \\
    (-1)^{n-1}\tbinom{n+1}{n-1}&\cdots&-\tbinom{n+1}{1}&\tbinom{n+1}{0}\\
  \end{array}\right)\left(\begin{array}{cccc}
 1&1&\cdots&1\\2^n&2^{n-1}&\cdots&2\\\vdots&\vdots&&\vdots\\n^n&n^{n-1}&\cdots&n
  \end{array}
  \right).
\end{equation}
\begin{remark}
Comparing Corollary~\ref{Cor:C-diagonalized} with Equ.~\eqref{Equ:Vandmonde}, Corollary~\ref{Cor:C-diagonalized} may be related to Miller's work studying the fraction of the Foulkes character table (see \cite[Corollary~1.1 and Example~1.2]{M15}). Unfortunately, we can not clarify the relationship.
 \end{remark}


\subsection*{Acknowledgements}The authors are grateful to the anonymous referee for her/his  numerous helpful comments and corrections, which greatly improved
the presentation of this paper.

\end{document}